\documentclass{amsart}
\usepackage{adjustbox,lipsum}
\usepackage{mathtools}
\usepackage{amsmath}
\usepackage{amssymb}
\usepackage[shortlabels]{enumitem}
\usepackage{colonequals}
\usepackage{tabu}
\usepackage{tikz-cd}
\usepackage{mathrsfs} 
\usepackage{makecell}

\usepackage[
        colorlinks, citecolor=darkgreen,
        backref,
        pdfauthor={Samir Siksek},
]{hyperref}

\usepackage{comment}

\newcommand{\PP}{\mathbb{P}}
\newcommand{\Q}{\mathbb{Q}}
\newcommand{\R}{\mathbb{R}}
\newcommand{\Z}{\mathbb{Z}}

\newcommand{\cP}{\mathcal{P}}

\newcommand{\cC}{\mathscr{C}}
\newcommand{\cX}{\mathscr{X}}
\newcommand{\cL}{\mathcal{L}}

\newcommand{\cF}{\mathcal{F}}

\newcommand{\vv}{\upsilon}

\newcommand{\fP}{\mathfrak{P}}

\newcommand{\OO}{\mathcal{O}}

\DeclareMathOperator{\Disc}{Disc}
\DeclareMathOperator{\re}{Re}

\DeclareMathOperator{\genus}{genus}

\DeclareMathOperator{\divv}{div}

\DeclareMathOperator{\Gal}{Gal}

\DeclareMathOperator{\Norm}{Norm}

\DeclareMathOperator{\ord}{ord}

\renewcommand{\setminus}{-}

\newtheorem{thm}{Theorem}
\newtheorem*{theorem}{Theorem}

\newtheorem{lem}[thm]{Lemma}
\newtheorem{question}[thm]{Question}
\newtheorem{conj}[thm]{Conjecture}
\newtheorem{cor}[thm]{Corollary}

\theoremstyle{definition}

\newtheorem{definition}[equation]{Definition}

\theoremstyle{remark}

\definecolor{darkgreen}{rgb}{0,0.5,0}

\begin{document}

\title[]{
New Algebraic Points on Curves
}

\begin{abstract}
Let $C/\Q$ be a smooth projective
        absolutely irreducible curve of genus $\ge 2$.
	For an extension $L/\Q$, let $C(L)_{\mathrm{new}} \colonequals
	\{P \in C(L) \; : \; \Q(P)=L\}$;
	this is the set of points on $C$ defined over $L$, but
	not over any strictly smaller field.
	Let $n \ge 2$.
        We conjecture that $C(L)_{\mathrm{new}}=\emptyset$ for $100\%$
	of degree $n$ number fields $L$, when ordered by absolute discriminant. 
	For degrees $n=2$, $3$, we give sufficient criteria for our conjecture
	to hold in terms of an explicit model for $C$. For general $n$ we prove
	a theorem that harmonises with the conjecture. 
 	In particular, we verify our conjecture in these cases
	\begin{itemize}
		\item $n=2$; $C=X_0(N)$ for the $18$ values $N \ne 37$
			such that $X_0(N)$ is hyperelliptic.
		\item $n=3$; $C=X_0(23)$, $X_0(29)$, $X_0(31)$, $X_0(64)$. 
	\end{itemize}
	Moreover, we prove the analogue of our conjecture for the unit equation,
	again with $n=3$.
\end{abstract}

\author{Maleeha Khawaja}

\address{Mathematics Institute\\
    University of Warwick\\
    CV4 7AL \\
    United Kingdom}

\email{maleeha.khawaja@warwick.ac.uk}

\author{Samir Siksek}

\address{Mathematics Institute\\
    University of Warwick\\
    CV4 7AL \\
    United Kingdom}

\email{s.siksek@warwick.ac.uk}

\date{\today}


\maketitle

\section{Introduction}
Let $C$ be a smooth projective and absolutely irreducible curve of genus $g$, defined over
a number field $K$.
Let $D$ be a reduced effective divisor
on $C$ and consider the \textbf{punctured curve} $X=C \setminus D$. The
Euler characteristic of $X$ is given by $\chi(X)=2-2g-\deg(D)$. We say
that $X$ is \textbf{hyperbolic} if $\chi(X)<0$. We write $\OO_K$
for the ring of integers of $K$, and we let $\cX/\OO_K$ be
a model for $X$ over $\OO_K$.
The famous Faltings--Siegel Theorem
asserts that $\cX(\OO_K)$ is finite. We mention three special cases.
\begin{itemize}[leftmargin=5mm]
\item If $D=0$ then $X=C$ is a complete curve, which is
hyperbolic if and only if $g \ge 2$. In this case, 
$\cX(\OO_K)=C(K)$ as $X$ is proper, and so Faltings--Siegel
generalises Faltings' theorem.
\item Let $X=\PP^1 \setminus \{0,1,\infty\}$, which
has Euler characteristic $-1$ and is therefore hyperbolic. It can
be shown that $\varepsilon \in \cX(\OO_K)$
if and only if $\varepsilon$ and $1-\varepsilon$ are units.
Thus Faltings--Siegel generalises Siegel's Theorem which states that 
the \textbf{unit equation} 
\begin{equation}\label{eqn:unit}
	\varepsilon+\delta=1, \qquad \varepsilon,~\delta \in \OO_K^\times
\end{equation}
has finitely many solutions.
\item Let $E/K$ be an elliptic curve specified by an affine Weierstrass model
\[
	y^2+a_1 xy + a_3 y \; = \; x^3+a_2 x^2+a_4 x +a_6, \qquad a_i \in \OO_K
\]
and let $\mathscr{O}$ denote the point at infinity. Let $X=E\setminus \mathscr{O}$,
which has Euler characteristic $-1$.
Then the Faltings--Siegel theorem applied to $\cX$ simply asserts
		that the affine model has finitely many points $(x,y) \in \OO_K^2$;
		this is Siegel's theorem for integral points on elliptic curves.
\end{itemize}
In the above cases, both the curve and the number field
are fixed. This paper is motivated by the following question. 
\begin{question}\label{question:big}
Given
a hyperbolic curve $\cX/K$ and a family $\cL$ of number field extensions $L/K$,
how does the set of integral points $\cX(\OO_L)$ 
vary with the number field?
\end{question}
A typical instance of the above question is 
when $X=C$  
is a complete curve of genus $g \ge 2$, and $\cL$ is the set of all
extensions $L/K$ of a fixed degree $n$. 
This instance of the question is well-considered in the literature;
for an excellent survey we recommend \cite{VirayVogt}, but
merely mention two famous results.
\begin{itemize}[leftmargin=5mm]
\item Merel's celebrated Uniform Boundedness Theorem
\cite{Merel}
asserts that all degree $n$ points on
the modular curve $X_1(p)/\Q$ (with $p$ prime)
are cuspidal, for $n<2 \log_{3}(\sqrt{p}-1)$.
\item A theorem of Harris and Silverman \cite{HS}
asserts that if a (complete) curve $C/\Q$ of genus $\ge 2$
has infinitely many quadratic points then 
$C$ is hyperelliptic or bielliptic.
This theorem builds on a deep theorem
of Faltings asserting that if $A/K$ is an abelian
variety and $V$ is a subvariety not
containing any abelian variety then $V(K)$ is finite.
\end{itemize}

In this paper we offer a conjectural statistical (partial) response to Question~\ref{question:big};
this is Conjecture~\ref{conj:diostab} stated below.
Let $X/K$ be a hyperbolic curve as before,
and let $L/K$ be an extension.
The conjecture is concerned with what we call the set of \textbf{$L$-new points} 
that appear over $L$ on $X/K$; we define these by
\[
\cX(\OO_L)_\mathrm{new}=\{ P \in \cX(\OO_L) \; : \; K(P)=L\}.
\]
We note that the $L$-new points are ones defined over $L$,
but not over any strictly smaller extension of $K$.
Before stating our conjecture, we explain what we mean by a family
$\cL$ of extensions $L/K$. 
\begin{definition}
	Let $K$ be a number field. Let $n \ge 2$.
	Let $G_1,\dotsc,G_r$ be transitive subgroups of $S_n$.
	Let $\vv_1,\dotsc,\vv_s$ be a finite number 
		of places of $K$, and let $L_i/K_{\vv_i}$
		be degree $n$ \'{e}tale algebras.
	By \textbf{the degree $n$ family $\cL$ with data $(G_1,\dotsc,G_r;L_{1},\dotsc,L_{s})$}
	we mean all extensions $L/K$ such that the following two
	conditions hold.
	\begin{itemize}
		\item	Write $\tilde{L}/K$ for the Galois closure 
			of $L/K$. Then
			$\Gal(\tilde{L}/K)$ 
	isomorphic to one of $G_1,\dotsc,G_s$.
	\item
	$L \otimes_{K} K_{\vv_i} \cong L_{i}$ for $i=1,\dotsc,r$.
	\end{itemize}
\end{definition}
We now state our conjecture, which we refer to as the
\textbf{statistical Diophantine stability conjecture for hyperbolic curves}.
\begin{conj}\label{conj:diostab}
Let $K$ be a number field and $\cX/\OO_K$ a model
for a hyperbolic curve. Let $\cL$ be a family of extensions $L/K$.
Then $\cX(\OO_L)_{\mathrm{new}}=\emptyset$ for 100\% of
$L \in \cL$, when ordered by norms of discriminants.
\end{conj}
If $X=C$ is a smooth projective absolutely irreducible curve over
a number field $K$, and $L/K$ is a finite extension, then 
\[
	\cX(\OO_L)_{\mathrm{new}} \; = \; C(L)_{\mathrm{new}} \; \colonequals
	\{P \in C(L) \; : K(P)=L\}.
\]
Thus Conjecture~\ref{conj:diostab} becomes the following simpler
statement.
\begin{conj}\label{conj:diostab2}
	Let $K$ be a number field and let $C/K$ be a smooth projective
	absolutely irreducible curve of genus $\ge 2$.
	Let $\cL$ be a family of extensions $L/K$.
	Then $C(L)_{\mathrm{new}}=\emptyset$ for $100\%$
	of $L \in \cL$, when ordered by norms of discriminants.
\end{conj}
\textbf{Remarks.}
\begin{itemize}
\item The idea of Diophantine stability first appears in a different
context due to Mazur and Rubin \cite{MazurRubin}.
\item It is tempting to conjecture that $\cX(\OO_L)=\cX(\OO_K)$
for $100\%$ of $L \in \cL$. This however is false.
For example, let $\cL$ be all quartic number fields $L/\Q$.
		Let $X/\Q$ be a hyperbolic curve such that $\cX(\Z)=\emptyset$,
		but $\cX(\Z[(1+\sqrt{5})/2]) \ne \emptyset$. A positive 
		proportion of quartic number fields $L$, when ordered by 
		absolute discriminant,
		contain $\Q(\sqrt{5})$
		(\cite[Corollary 1.2]{CDO}) and therefore $\cX(\OO_L) \ne \cX(\Z)$
		for those fields.
\end{itemize}

\medskip

We mention a few special cases of Conjecture~\ref{conj:diostab}
in the literature.
\begin{enumerate}[leftmargin=5mm]
\item[(I)] Let $n$ be a positive integer with $3 \nmid n$.
Let $\cL$ be the degree $n$ number fields $L/\Q$
where $3$ splits completely. A theorem of
Triantafillou \cite{Triantafillou} asserts that $(\PP^1 \setminus \{0,1,\infty\})(\OO_L)=\emptyset$
for $L \in \cL$. Similar theorems for many other \lq\lq locally restricted\rq\rq\
families
$\cL$ were proved  by
Freitas, Kraus and Siksek \cite{FKS, FKS2,FKS_PNAS, FKS_cyclic}.
\item[(II)] Let $\ell$ be a prime and let $\cL$ be the set of
cyclic degree $\ell$ number fields $L/\Q$.
Let $C/\Q$ be a curve and let $P \in C(\Q)$ be a rational point.
Write $J_C$ for the Jacobian of $C$ and suppose that $J_C(\Q)=0$.
Let $X=C\setminus \{P\}$. Siksek \cite{Siksek_puncture} shows that $\cX(\OO_L)=\emptyset$
for $100\%$ of $L \in \cL$, subject to a mild restriction on the
mod $\ell$ representation of $J_C$.
\end{enumerate}

\medskip

The purpose of this paper is to prove Conjecture~\ref{conj:diostab} for several
hyperbolic curves, with $\cL$ the set of all quadratic or cubic fields.

\begin{thm}\label{thm:hyp}
Let $g \in \Q[T]$ be squarefree, of even degree $d \ge 6$, and write $G=\Gal(g)$
for the Galois group of $g$. Suppose there is an element $\sigma \in G$
acting freely on the roots of $g$. 
Let $C/\Q$ be the hyperelliptic curve defined by $Y^2=g(T)$, and write $J$
for its Jacobian.
Suppose
	\begin{enumerate}[(a)]
		\item either $d \ge 10$;
		\item or $d=6$ and $J(\Q)$ is finite;
		\item or $d=8$ and $J(\Q)$ is finite or $J$ is simple.
	\end{enumerate}
Then $C(L)_{\mathrm{new}}=\emptyset$
for $100\%$ of quadratic fields $L/\Q$, when ordered by
absolute discriminant.
\end{thm}
We note that if $g$ is irreducible, then there is an element $\sigma \in G$
acting freely on its roots \cite{Jordan}. Naturally, as $C$ is hyperelliptic,
it has infinitely many quadratic points, but the above theorem is saying
that the quadratic fields over which they appear are relatively rare.

A famous theorem of Ogg \cite{Ogg} asserts that there are $19$ values of $N$ for which which $X_0(N)$ is hyperelliptic:
\begin{itemize}
\item genus $2$: $N=22$, $23$, $26$, $28$, $29$, $31$, $37$, $50$;
\item genus $3$: $N= 30$, $33$, $35$, $39$, $40$, $41$, $48$;
\item genus $4$: $N=47$;
\item genus $5$: $N=46$, $59$;
\item genus $6$: $N=71$.
\end{itemize}
\begin{cor}\label{cor:hyp}
	Let $N \ne 37$ be any of the above values for which $X_0(N)$
	is hyperelliptic.
	Then $X_0(N)(L)_{\mathrm{new}}=\emptyset$ for 
	$100\%$ of quadratic fields $L/\Q$, when ordered
	by absolute discriminant.
\end{cor}

\medskip

\begin{thm}\label{thm:cubic}
	Let $F(T,Y) \in \Q[T,Y]$ have degree $3$ in $Y$,
	and suppose $F$ is irreducible over $\overline{\Q}$.
	Let $C/\Q$ be the smooth projective curve with plane affine model $F(T,Y)=0$.
	Suppose $\genus(C) \ge 8$.
	Let $\Delta_Y(F) \in \Q[T]$
be the discriminant of $F$ with respect to $Y$, and write this as
\[
        \Delta_Y(F) \; = \; g(T) \cdot h(T)^2
\]
where $g$, $h \in \Q[T]$ with $g$ squarefree.
	Suppose that $g$ has even degree,
	and that there is an element $\sigma \in \Gal(g)$ acting freely on
	the roots of $g$. Then $C(L)_\mathrm{new}=\emptyset$ for $100\%$ of cubic
	fields $L/\Q$, when ordered by absolute discriminant.
\end{thm}
We note that the curve $C$ in Theorem~\ref{thm:cubic} has infinitely many
cubic points.
Theorems~\ref{thm:unit} and ~\ref{thm:modular} concern
cubic points on certain hyperbolic curves,
all of which have infinitely many cubic points. 
\begin{thm}\label{thm:unit}
	$(\PP^1 \setminus \{0,1,\infty\})(\OO_L)=\emptyset$ for $100\%$
	of cubic fields $L/\Q$, when ordered by absolute discriminant.
\end{thm}

\begin{thm}\label{thm:modular}
	Let $X=X_0(23)$, $X_0(29), X_0(31)$ or $X_0(64)$ (these have genera $2$, $2$, $2$, $3$ respectively).
	Then $X(L)_{\mathrm{new}}=\emptyset$ for $100\%$ of cubic fields $L/\Q$, when ordered by absolute discriminant.
\end{thm}
The four modular curves in Theorem~\ref{thm:modular}
do not satisfy the hypothesis of Theorem~\ref{thm:cubic}.
Nevertheless, the result follows from the proof of Theorem~\ref{thm:cubic}
together with further computations.

\medskip

Theorems~\ref{thm:cubic},~\ref{thm:unit} and~\ref{thm:modular} make use of results due to Bhargava, Taniguchi and
Thorne, 
recalled in
Section~\ref{sec:wood}, which count (by absolute discriminant)
cubic number fields
satisfying local restrictions. For number fields of general
degree, there are no such results yet. However, in 
certain settings we are able to say something about the discriminants of the primitive degree $n$ number fields
for which there are new points. Recall that a degree $n$ number field $L/\Q$
is primitive if its Galois group is a primitive subgroup of
$S_n$; equivalently, $L/\Q$ is primitive if the only subfields
$\Q \subseteq K \subseteq L$ are $K=\Q$ and $K=L$.

\begin{thm}\label{thm:primitive}
	Let $n \ge 2$. Let $F(T,Y) \in \Q[T,Y]$ have degree $n$ in $Y$,
	and suppose $F$ is irreducible over $\overline{\Q}$.
	Let $C/\Q$ be the smooth projective curve with plane affine model $F(T,Y)=0$.
	Suppose
	\begin{equation}\label{eqn:primitive}
		\genus(C) \; > \; n^2-n+1.
	\end{equation}
	Let $\Delta_Y(F) \in \Q[T]$
be the discriminant of $F$ with respect to $Y$, and write this as
\[
        \Delta_Y(F) \; = \; g(T) \cdot h(T)^2
\]
where $g$, $h \in \Q[T]$ with $g$ squarefree.
	Suppose that $g$ has even degree,
	and that there is an element $\sigma \in \Gal(g)$ acting freely on
	the roots of $g$. 
	Let $\cL$ be the set of all degree $n$ primitive number fields
	$L/\Q$. Let
	\[
		\Disc(C,\cL) = \{ \lvert \Delta(\OO_L) \rvert \; : \; 
		\text{$L \in \cL$ and $C(L)_{\mathrm{new}} \ne \emptyset$}\}.
	\]
	Then $\Disc(C,\cL)$ has density $0$ among the natural numbers. More precisely,
	there is some $0< \alpha<1$ such that 
\[
	\# (\Disc(C,\cL) \cap [1,X]) \; = \; O\left( \frac{X}{(\log{X})^\alpha}\right)
\]
as $X \rightarrow \infty$.
\end{thm}

\medskip

The paper is organized as follows.
Section~\ref{sec:count} counts, using
a theorem of Delange, integers
where the prime factors of the squarefree part belongs to certain
special subsets of primes. 
In Section~\ref{sec:single}
we consider $F(T,Y) \in \Q[T,Y]$ of degree $n$
in $Y$ and, under certain hypotheses,
we give an upper bound for the number
of discriminants of degree $n$ number fields
arising from polynomials $F(t,Y)$ with $t \in \Q$.
We apply this in Section~\ref{sec:primitive}, together with a theorem of Vojta,
to deduce Theorem~\ref{thm:primitive} and
Theorem~\ref{thm:hyp}. In Section~\ref{sec:wood} we recount
theorems due to
Bhargava, Taniguchi and
Thorne 
that count
cubic number fields with certain restrictions
on ramified primes; these together with techniques
introduced earlier in the paper allow us
to prove Theorem~\ref{thm:cubic}. 
The proofs 
of Theorems~\ref{thm:unit}
and ~\ref{thm:modular} are given in Sections~\ref{sec:unit}
and ~\ref{sec:modular} respectively, and make use
of results proved in Section~\ref{sec:wood} together with
explicit computations. Finally, in Section~\ref{sec:future}
we suggest research projects, based on the ideas of 
this paper, that some readers might be interested to pursue.

\medskip

We are grateful to Maarten Derickx and Victor Flynn for helpful discussions.

\section{Counting integers with certain restricted factorisations}\label{sec:count}
Let $\PP$ be the set of prime numbers
and let $\cP \subseteq \PP$. Following Serre \cite{Serre},
we call $\cP$ \textbf{regular of density $\alpha>0$}
if, for $\re(s)>1$,
\begin{equation}\label{eqn:regular}
\sum_{p \in \cP} \frac{1}{p^s} \; =\; \alpha \cdot \log\left( \frac{1}{s-1} \right) \; + \; \theta_A(s)
\end{equation}
where $\theta_A$ extends to a holomorphic function on $\re(s) \ge 1$.
We call the set $\cP$ \textbf{Frobenian of density $\alpha>0$}
if there exists a finite Galois extension $L/\Q$
and a subset $\cC$ of $G=\Gal(L/\Q)$, such that
\begin{itemize}
\item $\cC$ is a union of conjugacy classes in $G$;
\item $\alpha=\# \cC/\# G$;
\item for every sufficiently large prime $p$,
we have $p \in \cP$ if and only if $\sigma_p \in \cC$
where $\sigma_p$ is a Frobenius element of $G$
corresponding to $p$.
\end{itemize}
By the Chebotarev Density Theorem \cite[Proposition 1.5]{Serre},
if $\cP$ is Frobenian of density $\alpha>0$
then it is regular of density $\alpha>0$.

\begin{thm}\label{thm:count}
	Let $S$ be a set of primes of regular density $0 < \alpha<1$.
	Let $N_S(X)$ be the number of positive integers $n \le X$
	such that $\ord_p(n)$ is even for all $p \in S$. 
	Then, there exists some constant $C_S>0$ such that
	\[
		N_S(X) \; \thicksim \; C_S \cdot \frac{X}{(\log{X})^\alpha}.
	\]
\end{thm}
\begin{proof}
	Let 
	\[
		a_n=\begin{cases}
			1 & \text{if $\ord_p(n)$ is even for all $p \in S$}\\
			0 & \text{otherwise}.
		\end{cases}
	\]
	Then $N_S(X)=\sum_{n \le X} a_n$. We shall make use of a Tauberian theorem
	due to Delange \cite[page 350]{Tenenbaum} to estimate $N_S(X)$. 

	Consider the Dirichlet series
	\[
		D(s)=\sum_{n=1}^\infty \frac{a_n}{n^s},
	\]
	which defines a holomorphic function on $\re(s)>1$.
	We need to analyse the behaviour of $D(s)$ in the neighbourhood
	of $s=1$.
	For $\re(s)>1$ we write $D(s)$ as an Euler product
	\[
		\begin{split}
			D(s) & =\prod_{p \in S} \left(1+ \frac{1}{p^{2s}}+\frac{1}{p^{4s}}+\cdots \right)
		\cdot \prod_{p \notin S} \left(1+ \frac{1}{p^{s}}+\frac{1}{p^{2s}}+\cdots \right) \\
			& = \prod_{p \in S} \left(1-\frac{1}{p^{2s}}\right)^{-1}
			      \cdot \prod_{p \notin S}  \left(1-\frac{1}{p^{s}}\right)^{-1}.
		\end{split}
	\]
	Thus
	\[
		\log(D(s)) \; = \; \sum_{p \notin S} \frac{1}{p^s} \; + \; \theta(s)
	\]
	where $\theta$ is holomorphic at $s=1$. Since $\mathbb{P} \setminus S$
	has regular density $1-\alpha>0$, we conclude from \eqref{eqn:regular} that
	\[
		\log(D(s)) \; = \; (1-\alpha) \cdot \log \left(\frac{1}{s-1} \right)+\phi(s)	
	\]
	where $\phi$ is holomorphic at $s=1$. Hence
	\[
		D(s) \; = \; \frac{g_0(s)}{(s-1)^{1+\omega}}
	\]
	where $g_0(s)=\exp(\phi(s))$ is holomorphic and non-vanishing at $s=1$, and $\omega=-\alpha$; 
	the notation here is chosen to match that of Theorem II.7.28 of \cite{Tenenbaum}.
	We note that $-1<\omega < 0$. Applying the aforementioned theorem gives
	\[
		\sum_{n \le X} a_n \; \thicksim \; 
		\frac{g_0(1)}{\Gamma(\omega+1)} \cdot X(\log{X})^\omega.
	\]
	This gives the theorem with $C_S=\exp(\phi(1))/\Gamma(1-\alpha)$.
\end{proof}

\section{Degree $n$ points from a single source}
\label{sec:single}

To motivate what comes next, we mention the following
\lq\lq single source theorem\rq\rq\ due to the authors
\cite[Theorem 1.1]{KhawajaSiksek2}. In this context,
by a \textbf{primitive degree $n$ point} on $C/K$, we mean
an algebraic point $P \in C(\overline{K})$
such that $K(P)/K$ is a primitive extension of degree $n$.
\begin{theorem}\label{thm:single}
Let $K$ be a number field. Let $C/K$ be a curve of genus $g$,
and write $J$ for the Jacobian of $C$.
Let $n \ge 2$ and suppose
\begin{equation}\label{eqn:glb}
        \begin{cases} g>(n-1)^2 & \text{if $n \ge 3$}\\
                g\ge 3 & \text{if $n=2$}.
        \end{cases}
\end{equation}
Suppose that $A(K)$ is finite for every abelian subvariety $A/K$
of $J$ of dimension $ \le n/2$.
If $C$ has infinitely many primitive points of degree $n$,
then there is a degree $n$ morphism $\varphi : C \rightarrow \PP^1$
defined over $K$
such that all but finitely many primitive degree $n$ divisors
are fibres $\varphi^{*}(t)$ with $t \in \PP^1(K)$.
\end{theorem}
This theorem is called the \lq\lq Single Source Theorem\rq\rq\
since, with finitely many exceptions, all primitive degree $n$ points come
from a single source which is the morphism $\varphi : C \rightarrow \PP^1$.
A variant of this theorem is due to Derickx \cite{Derickx_Contract}.

In view of our Conjecture~\ref{conj:diostab2}, it makes sense
to consider the family of degree $n$ number fields we obtain
from the fibres of this morphism, and to attempt to show
that this family contains $0\%$ of degree $n$ number fields.
We shall give a result in this spirit.  

\medskip

Let $C$ be a smooth projective absolutely irreducible curve over $\Q$,
and let $\varphi : C \rightarrow \PP^1$ be a morphism
of degree $n \ge 2$.
Then the curve $C$ 
possesses a plane affine (possibly singular) model
\[
	C \; : \; F(T,Y)=0
\]
where $F(T,Y) \in \Q[T,Y]$ has degree $n$ in $Y$, is irreducible over $\overline{\Q}$,
and where the morphism $\varphi$ is given by the map $(T,Y) \mapsto T$. 
Let $\Theta \subset \Q$ be the set of $t \in \Q$
such that $F(t,Y)$ is either reducible, or has degree $<n$.
Hilbert's irreducibility theorem asserts that $\Theta$
is contained in a thin set.
Thus $C$ has infinitely many degree $n$ points. For $t \in \Q \setminus \Theta$
we write $L_t$ for the degree $n$ number field defined by $F(t,Y)=0$,
and let $\Delta(\OO_{L_t})$ be the discriminant of the ring of integers
$\OO_{L_t}$ of $L_t$.  We define the \textbf{discriminant set} of $F$
to be
\[
	\Disc_F \; = \; \{ \lvert \Delta(\OO_{L_t})\rvert \; : \; t \in \Q \setminus \Theta\}.
\]
We shall give a mild sufficient condition for the set $\Disc_F$ to
have natural density $0$ among the positive integers.
Whilst this is a long way off from Conjecture~\ref{conj:diostab2},
it does harmonise with it.
\begin{thm}\label{thm:sparse}
Let $\Delta_Y(F) \in \Q(T)$
be the discriminant of $F$ with respect to $Y$, and write this as
\[
	\Delta_Y(F) \; = \; g(T) \cdot h(T)^2
\]
where $g$, $h \in \Q[T]$, with $g$ squarefree.  
Suppose $g$ has even degree. Let $G=\Gal(g)$. Suppose that there is an element $\sigma \in \Gal(g)$
acting freely on the roots of $g$. Then there is some $0<\alpha<1$ such that, as $X \rightarrow \infty$, 
\[
	\# (\Disc_F \cap [1,X]) \; = \; O\left( \frac{X}{(\log{X})^\alpha}\right).
\]
In particular, the discriminant set $\Disc_F$ has density $0$ among the natural numbers.
\end{thm}

\begin{lem}\label{lem:sparse}
	Suppose $F$ satisfies the hypotheses
	of Theorem~\ref{thm:sparse}.
	Then there is a set $S$ of primes
having positive regular density, such that for all $t \in \Q \setminus \Theta$
and all $p \in S$, the valuation $\ord_p(\Delta(\OO_{L_t}))$ is even.
\end{lem}
\begin{proof}
Let $t \in \Q \setminus \Theta$, and write $t=a/b$ where $a$, $b$ are coprime integers.
We note that 
	\[
		\Delta(\OO_{L_t})/\Delta_Y(F(t,Y)) \; \in \; (\Q^\times)^2.
	\]
Thus
	\[
		\Delta(\OO_{L_t})/g(t) \; \in \; (\Q^\times)^2.
	\]
As $g$ has even degree, say $m$, we can write
\[
	g(t) \; = \; \frac{G(a,b)}{b^m}
\]
	where $G$ is homogeneous, and so
	\[
		\Delta(\OO_{L_t})/G(a,b) \; \in \; (\Q^\times)^2.
	\]
Thus, for any prime $p$,
	\[
		\ord_p(\Delta(\OO_{L_t})) \; \equiv \; \ord_p(G(a,b)) \pmod{2}.
	\]
By assumption, there is some $\sigma \in \Gal(g)$ acting freely on the roots of $g$. 
Therefore, by the Chebotarev Density Theorem, there is a set of primes $S$, having positive
	regular density, such that the congruence $G(u,v) \equiv 0 \pmod{p}$
	does not have solutions with coprime $u$, $v$. Hence for $p \in S$,
	we have $\ord_p(G(a,b))=0$, so $\ord_p(\Delta(\OO_{L_t}))$ is even.
\end{proof}

\begin{proof}[Proof of Theorem~\ref{thm:sparse}]
Theorem~\ref{thm:sparse} follows immediately from 
	Theorem~\ref{thm:count} and Lemma~\ref{lem:sparse}.
\end{proof}

\section{Proofs of Theorems~\ref{thm:hyp} and~\ref{thm:primitive}}\label{sec:primitive}

We shall need the following theorem of Vojta \cite[Corollary 0.3]{Vojta}.
\begin{theorem}[Vojta]
	Let $C$ be a curve defined over a number field $K$. Let $n$ be a positive
	integer and let $\varphi : C \rightarrow \PP^1$ be a non-constant morphism
	defined over $K$. Suppose
	\begin{equation}\label{eqn:Vojta}
		\genus(C)-1 \; > \; (n-1) \cdot \deg(\varphi).
	\end{equation}
	Then the set
	\[
		\{ P \in C(\overline{K}) \; : \; \text{$[K(P):K] \le n$ and $K(\varphi(P))=K(P)$} \} 
	\]
	is finite.
\end{theorem}

\begin{proof}[Proof of Theorem~\ref{thm:primitive}]
Let $C/\Q$ and $F(T,Y) \in \Q[T,Y]$ and $\cL$ be as in 
the statement of Theorem~\ref{thm:primitive}. In particular
$F$ has degree $n$ in $Y$. Let $\varphi : C \rightarrow \PP^1$
be defined on the affine patch $F(T,Y)=0$ by $(T,Y) \mapsto T$;
this has degree $n$. Let $L \in \cL$. Thus $L$
is a primitive degree $n$ number field. Let 
$P \in C(L)_{\mathrm{new}}$. Then $\Q(P)=L$.
Moreover, $\Q(\varphi(P)) \subseteq L$,
and therefore, as $L$ is primitive, $\Q(\varphi(P))=\Q$ or $L$.
We note that assumption~\eqref{eqn:primitive}
in Theorem~\ref{thm:primitive} implies \eqref{eqn:Vojta}.
It follows from Vojta's theorem that for all but
finitely many $L \in \cL$, any $(T,X) \in C(L)_{\mathrm{new}}$
satisfies $T \in \Q$. Thus, with at most finitely many
exceptions, the elements of $\Disc(C,\cL)$
belongs to $\Disc_F$. Theorem~\ref{thm:primitive}
follows immediately from Theorem~\ref{thm:sparse}.
\end{proof}

\begin{proof}[Proof of Theorem~\ref{thm:hyp}]
	Let $F(T,Y)=Y^2-g(T)$; the discriminant
	with respect to $Y$ is $\Delta_Y(F)=2^2 g(T)$.
	Let $d=\deg(g)$.
	Suppose first that $d \ge 10$.
	 Thus $C$ has genus $\ge 4$,
	and so \eqref{eqn:primitive} is satisfied.
	Theorem~\ref{thm:hyp} follows immediately
	from Theorem~\ref{thm:primitive}.

	We now consider $d=6$ and $d=8$ (and thus $C$ has genus $2$
	or $3$). We suppose that
assumptions (b), (c) in the statement of Theorem~\ref{thm:hyp} hold. 
	We invoke \cite[Proposition 18]{KhawajaSiksek},
	which is a standard 
	consequence of the proof by Faltings \cite{Faltings_Lang}
	of the Bombieri-Lang conjecture for subvarieties of abelian varieties.
	We conclude the existence of finitely many effective degree $2$
	divisors $D_1,\dots,D_m$ such that
	\[
		C^{(2)}(\Q)=\bigcup_{i=1}^m \lvert D_i \rvert
	\]
	where $\lvert D \rvert$ denotes the complete linear series
	corresponding to $D$. An easy application of the
	Riemann-Roch theorem and Clifford's theorem
	\cite[Theorem IV.5.4]{Hartshorne}
	shows that $\lvert D \rvert=\{D\}$ unless $D$ is a hyperelliptic
	divisor, 
	in which case 
	$\lvert D \rvert = \{ \varphi^*(t) : t \in \PP^1(\Q)\}$.
	Thus
	\[
		\{ P \in C(\overline{\Q}) \; : \; \text{$[\Q(P):\Q]=2$ and $\Q(\varphi(P))=\Q(P)$} \} 
	\]
	is finite. The theorem now follows from the proof of Theorem~\ref{thm:primitive}. 
\end{proof}

\begin{proof}[Proof of Corollary~\ref{cor:hyp}]
It is known, e.g. \cite{BruinNajman}, that $J_0(N)(\Q)$
	is finite for all $N$ such that $X_0(N)$ is hyperellitic,
	except for $N=37$. The hyperelliptic polynomials
	$g$ for these hyperelliptic modular curves are
	available in  the computer algebra package \texttt{Magma} \cite{Magma}. 
	We computed the Galois
	groups and checked that the hypotheses of
	Theorem~\ref{thm:hyp} apply.
\end{proof}
\noindent \textbf{Remark.} A recent paper of Najman and Trbovi\'c \cite{NajmanT} studies
primes that ramify in the field of definition of 
non-exceptional non-cuspidal quadratic points on hyperelliptic
modular curves $X_0(N)$ with $N \ne 37$. It is possible
to deduce Corollary~\ref{cor:hyp} from their results, together
with Theorem~\ref{thm:count}.

\section{Proof of Theorem~\ref{thm:cubic}}\label{sec:wood}
%
We shall need a theorem of 
Bhargava, Taniguchi and Thorne \cite[Theorem 1.3]{BTT}
that counts cubic fields, ordered by discriminant,
with local specifications. The following
two corollaries are easy consequences of that theorem.
\begin{cor}\label{cor:BTT1}
Let $T$ be a finite set of primes. Write $N_T(X)$ for the
number of isomorphism classes of cubic fields $L$ satisfying
	\begin{enumerate}[(i)]
\item $\lvert \Delta_L \rvert  \le X$;
\item $\ord_p(\Delta_L)=0$ for all $p \in T$.
\end{enumerate} 
Then
\[
N_T(X) \; = \; C_{T} \cdot X + O_T(X^{5/6})  
\]
where 
\[
	C_{T} \; =\;  \frac{1}{3 \zeta(3)} \cdot \prod_{p \in T} \left( 1 - \frac{p+1}{p^2+p+1}\right).
\]
\end{cor}
\begin{cor}\label{cor:BTT2}
Let $T$ be a finite set of primes, with $2$, $3 \notin T$. Write $N_T(X)$ for the
number of isomorphism classes of cubic fields $L$ satisfying
	\begin{enumerate}[(i)]
\item $\lvert \Delta_L \rvert  \le X$;
\item $\ord_p(\Delta_L)$ is even for all $p \in T$.
\end{enumerate} 
Then
\[
N_T(X) \; = \; C_{T} \cdot X + O_T(X^{5/6})  
\]
where
\[
	C_{T} \; =\;  \frac{1}{3 \zeta(3)}\cdot \prod_{p \in T} \left( 1 - \frac{p}{p^2+p+1}\right).
\]
\end{cor}
\begin{proof}[Proof of Corollaries~\ref{cor:BTT1} and~\ref{cor:BTT2}]
	Theorem 1.3 of \cite{BTT} expresses the density $C_T$ as a product
	of local densities, which depend on the splitting type of $p$ in $L$ (either totally split, partially split,
	inert, partially ramified, totally ramified). 
	In Corollary~\ref{cor:BTT1} the allowed splitting types for $p \in T$ are totally split, partially split,
	inert, and the corollary follows immediately from the aforementioned \cite[Theorem 1.3]{BTT}.

	We turn to the proof of Corollary~\ref{cor:BTT2}. For this we note that for $p \ne 2$, $3$, the valuation $\ord_p(\Delta_L)$ is either $0$, $1$ or $2$ according
to whether $p$ is unramified in $L$, partially ramified in $L$, or totally ramified in $L$; this follows 
	immediately from Lemma~\ref{lem:ramification} below.
	Thus, condition (ii) of Corollary~\ref{cor:BTT2} is equivalent to saying that for every $p \in T$ the splitting type
	of $p$ is either totally split, partially split, inert or totally ramified.
	The result again follows from \cite[Theorem 1.3]{BTT}.
\end{proof}

The following result is well-known. We include the proof for the convenience of the reader.
\begin{lem}\label{lem:ramification}
Let $L$ be a number field of degree $n$. Let $p>n$ be a rational prime
and write
	\[
		p\OO_L=\fP_1^{e_1}\cdots \fP_r^{e_r}
	\]
	where the $\fP_i$ are distinct prime ideals. Let $f_i$
	be the inertial degree of $\fP_i$. Then
	\[
		\ord_p(\Delta_L)=\sum_{i=1}^r f_i (e_i-1).
	\]
\end{lem}
\begin{proof}
The condition $p>n$ ensures that $p \nmid e_i$,
so the primes $\fP_i$ are tamely ramified in $L$.
	Write $\mathcal{D}_L$ for the different ideal of $L$. Then
	$\ord_{\fP_i}(\mathcal{D}_L)=e_i-1$
	(Theorem III.2.6 of \cite{Neukirch}).
	The lemma follows as $\Norm(\mathcal{D}_L)=\lvert \Delta_L \rvert$
	(Theorem III.2.9 of \cite{Neukirch}).
\end{proof}

\begin{lem}\label{lem:prod}
	Let $R=\{p_1,p_2,p_3,\dotsc\}$ be a set of primes having regular density $>0$.
	Let $f$, $g$ be monic polynomials with coefficients in $\R$, with $\deg(g)=\deg(f)+1$,
	and suppose $g(p_i) \ne 0$ for all $i$.
	Then 
	\[
		\lim_{k \rightarrow \infty} \prod_{i=1}^k \left(1 - \frac{f(p_i)}{g(p_i)}\right) =0.
	\]
\end{lem}
\begin{proof}
	It is sufficient to show that $\sum_{p \in R} f(p)/g(p)$ diverges. Suppose $\sum_{p \in R} f(p)/g(p)$ converges.
	We note that $xf(x)$ and $g(x)$ are monic polynomials of the same degree, therefore $g(p)/pf(p)$ is bounded, say by $B>0$,
	on $R$.  We note that $1/p  \le B \cdot f(p)/g(p)$ for $p \in R$. Thus $\sum_{p \in R} 1/p$ converges.
	Let $\delta$ be the regular density of $R$, which by assumption is positive.
	From \eqref{eqn:regular} we have
	\[
		\delta = \lim_{\sigma \rightarrow 1^+}\frac{\sum_{p \in R} \frac{1}{p^\sigma}}{-\log(\sigma-1)}.
	\]
	However, as $\sigma \rightarrow 1^+$,
	\[
		\frac{\sum_{p \in R} \frac{1}{p^\sigma}}{-\log(\sigma-1)}
	       \; \le \;
	       \frac{\sum_{p \in R} \frac{1}{p}}{-\log(\sigma-1)} \rightarrow 0.
        \]
	This contradicts
	the assumption that $\delta>0$.
\end{proof}

\begin{cor}\label{cor:discriminant}
	Let $S$ be a set of primes having regular density $<1$. Then the proportion
	of cubic number fields with discriminants supported only on $S$
	is $0\%$, when ordered by absolute discriminant.
\end{cor}
\begin{proof}
	Write $\cF(X)$ for the set of cubic fields with absolute discriminant
	at most $X$, and write $\cF_S(X)$ for the subset whose discriminants
	are supported on $S$. We want to show that
	\begin{equation}\label{eqn:limF}
		\lim_{X \rightarrow \infty} \frac{\# \cF_S(X)}{\# \cF(X)} =0.
	\end{equation}
	Let $p_1,\dotsc,p_k$ be primes belonging to $R:=\mathbb{P} \setminus S$. Then
	\[
		\cF_S(X) \; \subset \; \{ K \; : \; \text{$K \in \cF(X)$,
		$p_1,\dotsc,p_k$ unramified in $K$}\}.
	\]
	Hence
	\[
		0 \le \frac{\# \cF_S(X)}{\# \cF(X)} \le \frac{ N_{p_1,\dotsc,p_k}(X)}{N_{\emptyset}(X)}
		\rightarrow \prod_{i=1}^k \left( 1 - \frac{p_i+1}{p_i^2+p_i+1}\right)
	\]
	as $X \rightarrow \infty$,
	where the limit follows from Corollary~\ref{cor:BTT1}. 
	From Lemma~\ref{lem:prod}, by choosing $p_1,\dotsc,p_k$ be a suitable 
	(finite) subset of $R$, we can ensure that the product on the right hand-side is as small 
	as we like. This completes the proof.
\end{proof}
\begin{cor}\label{cor:discriminant2}
	Let $S$ be a set of primes having regular density $>0$. 
		Then the proportion of cubic fields
			with discriminants having even valuation at all $p \in S$
			is $0$.
\end{cor}
\begin{proof}
	Write $\cF(X)$ for the set of cubic fields with absolute discriminant
        at most $X$, and write $\cF_S(X)$ for the subset whose discriminants have even valuation
	at all $p \in S$.
	Again we want to prove \eqref{eqn:limF}. 
        Let $p_1,\dotsc,p_k$ be primes belonging to $S$. Then
        \[
                \cF_S(X) \; \subset \; \{ K \; : \; \text{$K \in \cF(X)$,
		where $\ord_{p_i}(\Delta_K)$ is even for all $i=1,\dotsc,k$}\}.
        \]
        Hence
        \[
                0 \le \frac{\# \cF_S(X)}{\# \cF(X)} \le \frac{ N_{p_1,\dotsc,p_k}(X)}{N_{\emptyset}(X)}
                \rightarrow \prod_{i=1}^k \left( 1 - \frac{p_i}{p_i^2+p_i+1}\right)
        \]
        as $X \rightarrow \infty$,
	where the limit follows from Corollary~\ref{cor:BTT2}.
        From Lemma~\ref{lem:prod}, by choosing $p_1,\dotsc,p_k$ be a suitable
        subset of $S$, we can ensure that the product on the right hand-side is as small
        as we like. This completes the proof.
\end{proof}

\begin{cor}\label{cor:sparsecubics}
Let $F(T,Y) \in \Q[T,Y]$ be irreducible and have degree $3$ in $Y$. 
	Let $\Theta$ be the set of $t \in \Q$ such that $F(t,Y)$ 
	either has degree $<3$ or is reducible.
Let $\Delta_Y(F) \in \Q[T]$
be the discriminant of $F$ with respect to $Y$, and write this as
\[
        \Delta_Y(F) \; = \; g(T) \cdot h(T)^2
\]
where $g$, $h \in \Q[T]$ with $g$ squarefree.
Suppose $g$ has even degree. Let $G=\Gal(g)$. Suppose that there is an element $\sigma \in \Gal(g)$
acting freely on the roots of $g$. Let $\cF^\prime$ be the set of cubic fields
we obtain from $F(t,Y)=0$ with $t \in \Q \setminus \Theta$.
Then $100\%$ of cubic fields, ordered by discriminant,
do not belong to the family $\cF^\prime$.
\end{cor}
\begin{proof}
	By Lemma~\ref{lem:sparse},
	there is a set $S$ of primes
having positive regular density, such that for all $t \in \Q \setminus \Theta$
and all $p \in S$, the valuation $\ord_p(\Delta(\OO_{L_t}))$ is even, for any
	$L_t \in \cF^\prime$. The desired result follows from Corollary~\ref{cor:discriminant2}.
\end{proof}

\begin{proof}[Proof of Theorem~\ref{thm:cubic}]
	We apply the proof of Theorem~\ref{thm:primitive} with $n=3$. We deduce that
	all but finitely many cubic points on $C$ arise as fibres
	of the map $C \rightarrow \PP^1$ given on the affine model
	$F(T,Y)=0$ by $(T,Y) \mapsto T$; the assumption $\genus(C) \ge 8$
	ensures that $C$ satisfies condition \eqref{eqn:primitive}.
	Theorem~\ref{thm:cubic} follows immediately from
	Corollary~\ref{cor:sparsecubics}.
\end{proof}

\section{Proof of Theorem~\ref{thm:unit}}\label{sec:unit}
Let $L$ be a number field. Nagell \cite{Nagell}
calls a unit $\varepsilon \in
\OO_L^\times$ \textbf{exceptional} if $1-\varepsilon \in \OO_L^\times$.
The number field $L$ is called \textbf{exceptional} if it possesses
an exceptional unit. Thus $\varepsilon$ is exceptional
if and only if $(\varepsilon,1-\varepsilon)$ is a solution to the
unit equation \eqref{eqn:unit}, and $L$ is exceptional if and only if the unit
equation has solutions in $L$. 
Nagell \cite{Nagell1928},
\cite{Nagell2}, has determined all exceptional number fields
where the unit rank is $0$ or $1$. Nagell finds
\cite[Section 2]{Nagell2}
that the only exceptional quadratic fields are $\Q(\sqrt{5})$
and $\Q(\sqrt{-3})$ which contain
exceptional units $(1+\sqrt{5})/2$ and $(1+\sqrt{-3})/2$ respectively,
and the only exceptional complex cubic
fields are the ones with discriminants $-23$ and $-31$.
He also showed \cite[Sections 3--5]{Nagell2}
that the only exceptional real cubic fields 
are of the form $\Q(\varepsilon)$ where $\varepsilon$ is a root of
\[
f_t(X)=X^3+(t-1)X^2-tX-1, \qquad t \in \Z, \quad t \ge 3
\]
or of
\[
g_t(X)=X^3+tX^2-(t+3)X+1, \qquad t \in \Z, \quad t \ge -1,
\]
and in both cases $\varepsilon$ is an exceptional unit.
Whilst the number of exceptional cubic fields is infinite,
we are able to show the following.
\begin{thm}
	\label{thm:exceptionalcubic}
	$100\%$ of cubic number fields are non-exceptional, 
	when ordered by absolute 
	discriminant.
\end{thm}
This theorem is of course equivalent to Theorem~\ref{thm:unit}.

\begin{proof}
	By the aforementioned work of Nagell~\cite{Nagell2}, it remains to prove 
	that $100\%$ of \emph{real} cubic fields are non-exceptional, when ordered by absolute discriminant. 
	We will first consider the case that $\varepsilon$ is an 
	exceptional cubic unit and a root of the polynomial
	\begin{equation*}
        g_t(X)=X^3+tX^2-(t+3)X+1
	\end{equation*}
	for some $t\in\Z_{\geq-1}$.
	The discriminant of $g_{t}(X)$ is  
	$(t^2 + 3t + 9)^2$, and therefore $\Q(\varepsilon)$ 
	has Galois group $C_3$.
	This completes the proof in this case as 
	Cohn and Davenport--Heilbronn have shown that 
	$100\%$ of cubic fields have Galois group $S_{3}$, 
	when ordered by absolute discriminant \cite{Cohn, DavenportHeilbronn}. 
	Now suppose $\varepsilon$ is an exceptional cubic unit and a root of the polynomial
	\begin{equation*}
		f_{t}(X)=X^3 + (t - 1)X^2 - tX - 1
	\end{equation*}
	for some $t\in\Z_{\geq 3}$.
	The discriminant of $f_{t}(X)$, with respect to $X$, is given by 
	\[
	F(t)\colonequals \Delta_{X}(f_{t})=t^4 + 6t^3 + 7t^2 - 6t - 31.
	\] 
	It follows that if a prime $p$ divides the discriminant of $\OO_{\Q(\varepsilon)}$
	then it divides $F(t)$ and hence the polynomial $F$
	has a root modulo $p$. Write $S$ for the set of primes $p$
	such that $F$ has a root modulo $p$. 
	The Galois group of $F$ is $D_4$ and an easy
	application of the Chebotarev Density Theorem shows that $S$ has Frobenian (and hence regular) density $3/8$.
	
	Write $\cF_S$ for the set of cubic fields with discriminants supported on $S$.
	By the above, the real exceptional fields with Galois group $S_3$ all belong to $\cF_S$.
	The latter has density $0$ by Corollary~\ref{cor:discriminant}. This completes
	the proof.
\end{proof}

\section{Proof of Theorem~\ref{thm:modular}}\label{sec:modular}
We will give the details for $X_0(64)$, which has genus $3$. The other cases are similar.
The paper of Ozman and Siksek \cite{OzmanSiksek} gives the following convenient
model for $X_0(64)$ as a plane quartic
\[
	X_0(64) \; : \; x^3 z + 4 x z^3 - y^4 \; = \; 0.
\]
We point in passing that this is isomorphic to the Fermat quartic, and 
that 
Bremner and Choudhry \cite[Theorem 3.1]{Bremner20}
have shown that all cubic points on this curve have Galois
group $S_{3}$. We shall parametrize the cubic points.
We note the following four rational points on the above model:
\[
P_1=[0:0:1],\; P_2=[2:2:1], \; P_3:=[2:-2:1], \; P_4:=[1:0:0].
\]
Let 
\[
D_1 =-P_1+3P_2-3 P_3+P_4  ,\;  D_2=-P_1+2 P_2-2 P_3+P_4 ,\; D_3=2 P_1-2 P_2+P_3-P_4 .
\]
Using the Magma programs accompanying the paper \cite{OzmanSiksek} we find that
\[
	J(\Q)=(\Z/2\Z) \cdot [D_1]+(\Z/4\Z) \cdot [D_2] + (\Z/4\Z) \cdot [D_3].
\]
Now let $Q$ be a degree $3$ point on $X_0(64)$, and let $D$ be the corresponding
irreducible effective degree $3$ divisor. Then $D-3P_1$ is a degree $0$ divisor,
and so $[D-3P_1] \in J(\Q)$. Thus,
\[
	D \sim \underbrace{3P_1+a D_1+b D_2+c D_3}_{D_{a,b,c}}, \qquad 0 \le a \le 1, \quad -1 \le b, c \le 2.
\]
Using Magma \cite{Magma} we computed the Riemann--Roch space $L(D_{a,b,c})$ for the $32$ possible combinations $(a,b,c)$.
For all but $4$ of these, the Riemann--Roch dimension is $1$, and therefore there is precisely one effective $D$ linearly equivalent to $D_{a,b,c}$.
For the remaining $4$, namely $(a,b,c)=(0,0,0)$, $(0,1,0)$, $(0,1,1)$, $(1,2,1)$, 
the Riemann--Roch dimension is $2$, and therefore there are infinitely many effective
degree $3$ divisors linearly equivalent to $D_{a,b,c}$. Consider the case $(a,b,c)=(0,0,0)$.
Then $D_{(0,0,0)}=3P_1$, and a 
basis for $L(D_{(0,0,0)})$ is given by $1$, $\varphi \colonequals y/x$,
where these
are considered as elements of the function field $\Q(X_0(64))$. Thus 
\[
	D=\divv(\varphi-t)+3P_1
\]
for some $t \in \Q$. We can also consider $\varphi$ as a degree $3$ morphism $X_0(64) \rightarrow \PP^1$,
and then $D=\varphi^{-1}(t)$. This gives an infinite family of effective degree $3$ divisors on $X_0(64)$. More concretely,
write $T=y/z$, and $Y=x/z$ and let
\[
	F(T,Y)=4Y^3+Y-T^4;
\]
then $F(T,Y)=0$ is an affine plane model for $X_0(64)$ where the morphism 
$\varphi$ is simply given by $(T,Y) \mapsto T$.
In the notation of Corollary~\ref{cor:sparsecubics}, we find
\[
	\Delta_Y(F)= -432 T^8 - 16 = g(T) \cdot h(T)^2, \qquad g(T)=-27 T^8 - 1, \quad h(T)=4.
\]
As $g$ is irreducible, the Galois group will contain elements that act freely on the roots by a theorem of Jordan~\cite{Jordan}.
More precisely, the Galois group is $\Gal(g)=D_8 \rtimes C_2$ which has order $32$, and $21$ of the elements act freely on the roots.
By Corollary~\ref{cor:sparsecubics}, $100\%$ of cubic fields will not
arise from equations of the form $F(t,Y)=0$ with $t \in \Q$.
The other three cases $(a,b,c)=(0,1,0)$, $(0,1,1)$, $(1,2,1)$
are similar. 
We refer the reader to 
\[
	\text{\url{https://github.com/MaleehaKhawaja/NewPoints}} 
\]
for the Magma code supporting the proof of this theorem.

\section{Suggested Future Projects}\label{sec:future}
We invite the reader to help gather further evidence towards our Conjecture~\ref{conj:diostab}.
Here are some projects that we believe are \lq low-hanging fruit\rq:
\begin{enumerate}[(I)]
	\item Theorem~\ref{thm:primitive} gives an upper bound for the number of
	discriminants of degree $n$ primitive number fields coming from
	degree $n$ new points on a curve $C$. It assumed that $C$
	has a plane model of the form $F(T,Y)=0$ where $F$ has degree $n$
		in $Y$. This assumption is equivalent to the existence
		of a degree $n$ morphism $\psi : C \rightarrow \PP^1$.
	We expect that it should be possible to prove a similar theorem
		where there is a degree $n$ morphism $\psi : C \rightarrow E$ to an elliptic curve $E$.
		Here, instead of Vojta's 
		Theorem (recounted in Section~\ref{sec:primitive}),
		it would be useful to apply a more general theorem
		due to Song and Tucker \cite[Proposition 2.3]{SongTucker}.
	\item The proof of Theorem~\ref{thm:cubic} makes use of a theorem
		of Bhargava, Taniguchi and Thorne \cite[Theorem 1.3]{BTT}
that counts cubic fields, ordered by discriminant,
with local specifications. Similar theorems for degrees $4$ and $5$
		have been established by Ellenberg, Pierce and Wood \cite[Theorems 4.1 and 5.1]{EPW}.
		As far as we can tell, to prove the analogue of Theorem~\ref{thm:cubic}
		for degrees $4$ and $5$ would require a refinement of the aforementioned
		theorems of Ellenberg, Pierce and Wood, that consider densities
		of different possible ramification types for primes.
	\item Corollary~\ref{cor:hyp} and Theorem~\ref{thm:modular} establish
		Conjecture~\ref{conj:diostab2} in degrees $2$, $3$
		for certain $X_0(N)$. It should be possible
		to establish many similar results for other modular curves
		appearing in the $\mathrm{L}$-Functions and Modular Forms database \cite{lmfdb:modular}.
\end{enumerate}

\bibliographystyle{abbrv}
\bibliography{Primitive}
\end{document}